\documentclass[10pt]{amsart}
\usepackage{amssymb,amsthm,amsmath,amsfonts}
\usepackage{hyperref}
\def\R{{\mathbb R}}

\def\Z{{\mathbb Z}}
\def\T{{\mathbb T}}

\def\<{\langle}
\def\>{\rangle}

\def\S{{\mathbb S}}

\def\X{{\mathbb X}}

\def\Y{{\mathbb Y}}

\newtheorem{th-def}{Theorem-Definition}[section]
\newtheorem{theo}{Theorem}[section]

\newtheorem{prop}[theo]{Proposition}

\newtheorem{cor}[theo]{Corollary}

\title{A Characterization of the Normal Distribution
by the Independence of a Pair of Random Vectors}
\author[Wiktor Ejsmont]{Wiktor Ejsmont}

\address[Wiktor Ejsmont]{
 Department of Mathematics and Cybernetics
  Wroc\l aw University of Economics \\
ul. Komandorska 118/120, 53-345 Wroc³aw, Poland}
\email{wiktor.ejsmont@ue.wroc.pl}
\subjclass[2010]{Primary: 46L54. Secondary: 62E10.}

\keywords{Cumulants, characterization of normal distribution}
\begin{document}
\begin{abstract} 
 Kagan and Shalaevski \cite{KaganShalaevski} have shown that if the random variables  $\X_1,\dots,\X_n$ are  independent and identically distributed  and the  distribution of $\sum_{i=1}^n(\X_i+a_i)^2$ $a_i\in \mathbb{R}$ depends only on $\sum_{i=1}^na_i^2$ , then each $\X_i$ follows the normal distribution $N(0, \sigma)$. Cook \cite{Cook} generalized this result  replacing independence of all $\X_i$ by the independence of  $(\X_1,\dots, \X_m) \textrm{ and } (\X_{m+1},\dots,\X_n )$ and removing the requirement that $\X_i$ have the same distribution.
In this
paper, we will give  other characterizations of the normal distribution which are formulated in a similar spirit.
\end{abstract}
\maketitle

\section{Introduction}
\begin{quote}
\textit{It will be shown that the formulae are much
simplified by the use of cumulative moment
functions, or semi-invariants, in place of the
crude moments.
R.A. Fisher \cite{Fisher}}.
\end{quote}

The original motivation for this paper comes from a desire to understand the results about  characterization of normal distribution which were shown in \cite{Cook} and \cite{KaganShalaevski}. They proved, that the characterizations of a normal law  are given by a certain
invariance of the  noncentral chi-square distribution. 
It is a known fact that if $\X_1,\dots,\X_n$ are i.i.d. and following the normal distribution $N(0,\sigma)$ then the distribution of the statistic $\sum_{i=1}^n(\X_i+a_i)^2$, $a_i\in \mathbb{R}$ depends on $\sum_{i=1}^na_i^2$ only (see  \cite{Bryc1,Morgan}). Kagan and Shalaevski \cite{KaganShalaevski} have shown that if the random variables 
$\X_1 ,\X_2 ,..., \X_n$ are independent and identically distributed and the distribution
of $\sum_{i=1}^n(\X_i+a_i)^2$ depends only on $\sum_{i=1}^na_i^2$, then each $\X_i$ is normally distributed as $N(0, \sigma)$. Cook generalized this result  replacing independence of all $\X_i$ by the independence of  $(\X_1,\dots, \X_m) \textrm{ and } (\X_{m+1},\dots,\X_n )$ and removing the requirement that $\X_i$ have the same distribution. The theorem proved below gives a new look on this subject, i.e. we will show that in the  statistic $\sum_{i=1}^n(\X_i+a_i)^2=\sum_{i=1}^n\X_i^2 +2\sum_{i=1}^n\X_ia_i+ \sum_{i=1}^na_i^2$ only the linear part $\sum_{i=1}^n \X_ia_i$ is important. 
 In particular, from the above result we get  Cook Theorem from \cite{Cook},  but under the assumption  
that all moments exist. Note that Cook does not assume any moments, but he gets this result under integrability assumptions imposed on the corresponding random variable.  This paper is removing or at least relaxing its integrability assumptions.

The paper is organized as follows. In section 2 we review basic facts about cumulants. Next in the third section we   state and prove the main results (proposition). In this section we also discuss   the problem. 

\section{Cumulants and moments}

Cumulants were first defined and studied by the Danish scientist T. N. Thiele.
He called them semi-invariants. The importance of cumulants comes from
the observation that many properties of random variables can be better
represented by cumulants than by moments. We refer to Brillinger \cite{Brillinger} and
Gnedenko and Kolmogorov \cite{G-K} for further detailed probabilistic aspects
of this topic.

Given a random variable $\X$ with the moment generating function $g(t)$, its
$i$th cumulant $r_i$ is defined as
$$r_i(\X):=r_{i}(\underbrace{ \X,\dots,\X}_{i-times})=\frac{d^i}{dt^i}\Big|_{t=0}log( g(t)).$$
That is, 
$$\sum_{i=0}^\infty\frac{m_i}{i!}t^i=g(t)=exp\Big(\sum_{i=1}^\infty\frac{r_i}{i!}t^i\Big)$$
where $m_i$ is the $i$th moment of $\X$.

Generally, if $\sigma$ denotes the standard deviation, then
$$r_1=m_1,\qquad r_2=m_2-m_1^2=\sigma, \qquad r_3=m_3-3m_2m_1+2m_1^3.$$
The joint cumulant of several random variables $\X_1, \dots \X_n$ of order $(i_1,\dots,i_n)$, where $i_j$ are nonnegative integers, is defined by a similar  generating function $g(t_1,\dots,t_n)=E\big(e^{\sum_{i=1}^nt_i\X_i}\big)$
$$r_{i_1+\dots +i_n}(\underbrace{ \X_1,\dots,\X_1}_{i_1-times},\dots,\underbrace{ \X_n,\dots,\X_n}_{i_n-times})=\frac{d^{i_1+\dots +i_n}}{dt_1^{i_1}\dots dt_n^{i_n}}\Big|_{t=0}log( g(t_1,\dots,t_n)),$$
where $t=(t_1,\dots,t_n).$ 
\\ \\
\noindent Random variables $ \mathbb{X}_{1},\dots ,\mathbb{X}_{n} $  are  independent if and only if, for every $n \geq 1$ and every non-constant choice of $\mathbb{Y}_{i} \in \{ \mathbb{X}_{1},\dots ,\mathbb{X}_{n}  \}$, where $i \in \{1,\dots,k\}$ (for some positive integer  $k\geq 2$) we get $r_{k}( \mathbb{Y}_{1},\dots ,\mathbb{Y}_{k} )=0$. 
\\ \\
Cumulants of some important and familiar random distributions are listed
as follows:
\begin{itemize}
\item The Gaussian distribution $N(\mu,\sigma)$ possesses the simplest list of cumulants: $r_1=\mu$, $r_2=\sigma$ and  $r_n=0$ for $n\geq 3 $,
\item for the Poisson distribution with mean $\lambda$ we have $r_n=\lambda$.
\end{itemize}
These classical examples clearly demonstrate the simplicity and efficiency
of cumulants for describing random variables. 
Apparently, it is not accidental that cumulants  encode the most important information of the associated
random variables. The underlying reason may well reside in the following
three important properties (which are in fact related to each other):
\begin{itemize}
\item
(Translation Invariance) For any constant $c$,
$r_1(\X+c)=c+r_1(\X)$ and  $r_n(\X+c)=r_n(\X)$,  $n\geq 2$.
\item (Additivity) Let $\X_1,\dots, \X_m$ be any independent random variables.
Then, $r_n(\X_1+\dots+\X_m)=r_n(\X_1)+\dots+r_n(\X_m)$, $n\geq 1$.
\item  (Commutative property)  $r_n(\X_1,\dots,\X_n)=r_n(\X_{\sigma(1)},\dots,\X_{\sigma(n)})$ for any permutation $\sigma\in S_n.$
\item (Multilinearity) $r_{k}$ are the $k$-linear maps. 
\end{itemize}
For more details about cumulants and
 probability theory, the reader can consult \cite{Lehner} or \cite{Rota} .

\section{The Characterization theorem}
The main result of this paper is the following characterization of normal distribution
in terms of independent random vectors. 

\begin{prop}
Suppose vectors $(\S_1,\Y) \textrm{ and } (\S_2,  \Z )$ with all moments are independent and $\S_1,\S_2$ are
nondegenerate. If for every $a,b \in \mathbb{R}$ the linear combination $a\S_1 + \Y + b\S_2 + \Z$ has the law that
depends on $(a, b)$ through $a^2 + b^2$ only, then random variables $\S_1, \S_2$ have the same normal distribution 
and $cov(\S_1,\Y)=cov(\S_2,\Z)=0$.

\label{prop:1}
\end{prop}
\begin{proof}
Let $h_k(a^2+b^2)=r_k(a\S_1+\Y+b\S_2+\Z)-r_k(\Y+\Z)$. 
Because of the independence of $(\S_1,\Y) \textrm{ and } (\S_2,  \Z )$  we may write 
\begin{align} & h_k(a^2+b^2)=r_k(a\S_1+\Y+b\S_2+\Z)-r_k(\Y+\Z)=r_k(a\S_1+\Y)-r_k(\Y)+r_k(b\S_2+\Z)-r_k(\Z). \label{eq:niezalensoc} \end{align}
Evaluating \eqref{eq:niezalensoc} first when $b=0$ and then when $a=0$, we get
$h_k(a^2)=r_k(a\S_1+\Y)-r_k(\Y)$ and $h_k(b^2)=r_k(b\S_2+\Z)-r_k(\Z)$, respectively. Substituting this into \eqref{eq:niezalensoc}, we see

 $$h_k(a^2+b^2)=h_k(a^2)+h_k(b^2).$$ 

Note that $h_k(u)$ is continuous in $u\in[0,\infty)$, which implies $h_k(u)=h_k(1)u$ and so we have 
$h_k(a^2+b^2)=(a^2+b^2)h_k(1)=(a^2+b^2)\varphi_k(\alpha,\beta)$, where $$\varphi_k(\alpha,\beta)=h_k(\alpha^2+\beta^2)=r_k(\alpha\S_1+\Y+\beta\S_2+\Z)-r_k(\Y+\Z), $$ 
with $\alpha^2+\beta^2=1$.  In the next part of the proof, we will  compare polynomial,  which give us the correct
cumulants values. 
Let's first consider
the following equation $$h_k( a^2)= a^2\varphi_k(1,0),$$ which gives us 

\begin{align} r_k(a\S_1+\Y)-r_k(\Y)=a^2(r_k(\S_1+\Y)-r_k(\Y)). \label{eq:1} \end{align}
For $k=1$ we get  $E(\S_1)=0$, because $r_1(a\S_1+\Y)=E(a\S_1+\Y)$.
By putting $k=2$ in \eqref{eq:1} and using $r_2(a\S_1+\Y)=Var(a\S_1+\Y)=a^2Var(\S_1)+2acov(\S_1,\Y)+Var(\Y)$ we see  $$2a cov( \S_1, \Y)+a^2 Var( \S_1)=a^2(2cov( \S_1, \Y)+Var( \S_1)), $$ for all $a\in\R$, which implies that    $cov( \S_1, \Y)=0$.
Now by expanding equation \eqref{eq:1}  ($r_k$ are $k$-linear maps, we also  use independence),
we may write 
\begin{align} \sum_{i=1}^ka^{i} {k \choose i} r_k(\underbrace{ \S_1,\dots,\S_1}_{i-times},\underbrace{ \Y,\dots,\Y}_{k-i-times})=a^2(r_k(\S_1+\Y)-r_k(\Y)), \label{eq:2} \end{align}
for $k\geq 2$.
 This gives us  $r_k(\S_1)=0$ for $k>2$ and we have actually proved that $\S_1$ have the normal distribution  with zero mean.  Analogously, we will show that   $cov( \S_2, \Z)=0$ and  normality of  $\S_2.$ 
\\
\\
\noindent 
The next example presents an analogous construction for $h_k( a^2)= a^2\varphi_k^{0,1}(1)$, which involves the element $\varphi_k^{0,1}(1)$ instead of  $\varphi_k^{1,0}(1)$ which leads to $2a cov( \S_1, \Y)+a^2 Var( \S_1)=a^2(2cov( \S_2, \Z)+Var( \S_2)) $. But in the previous paragraph we calculated that $cov( \S_1, \Y)=cov( \S_2, \Z)=0$ which means that $Var( \S_1)=Var( \S_2)$ (we have common variance), i.e. we have the same distribution. 

\end{proof}
\noindent
As a corollary we get the following theorem.
\begin{theo}
Let $(\X_1,\dots, \X_m,\Y) \textrm{ and } (\X_{m+1},\dots,\X_n,\Z)$ be independent  random vectors with all moments, where  $\X_i$ are
nondegenerate,  and let statistic $\sum_{i=1}^na_i\X_i+\Y+\Z$
have a distribution which depends only on $\sum_{i=1}^n a_i^2$, where $a_i\in \mathbb{R}$ and $1\leq m < n$.  Then $\X_i $ are independent and have the same normal distribution with zero means and $cov(\X_i,\Y)=cov(\X_i,\Z)=0$ for $i\in\{1,\dots,n\}$. 
\label{tw:Main1}
\end{theo}
\begin{proof}
Without loss of generality, we may assume
that $m\geq 2.$
If we put \begin{align} \S_1=\frac{\sum_{i=1}^m a_i\X_i}{\sqrt{\sum_{i=1}^m a_i^2}} \textrm{ and }\S_2=\frac{\sum_{i=m+1}^n a_i\X_i}{{\sqrt{\sum_{i=m+1}^n a_i^2}}}\end{align} and $a=\sqrt{\sum_{i=1}^m a_i^2}$, $b=\sqrt{\sum_{i=m+1}^n a_i^2}$, in Proposition \ref{prop:1} then we get  that the distribution of
$$\S_1a+\Y+\S_2b+\Z=\sum_{i=1}^m a_i\X_i+\Y+\sum_{i=m+1}^na_i\X_i+\Z,$$
depends only on $a^2+b^2=\sum_{i=1}^n a_i^2$ ,which by Proposition  \ref{prop:1} implies that $\sum_{i=1}^m a_i\X_i$ 
have the normal distribution and $cov(\sum_{i=1}^m a_i\X_i,\Y)=0$ for all $a_i\in\R$.  Now, we once again use Proposition  \ref{prop:1} with   $\S_1=\X_1, \textrm{  }\S_2=\X_{m+1},$ 
then we see from assumption that the distribution of
$$a_1\X_1+\X_2+\Y+a_{m+1}\X_{m+1}+\Z$$
(where $\X_2+\Y$ play a role of $\Y$ from Proposition  \ref{prop:1}), depends only on $ a_1^2+a_{m+1}^2$ ($ a_1^2+a_{m+1}^2+1$). This gives us $cov(\X_1,\X_2+\Y)=0$, but we know  that $cov( \X_1,\Y)=0$ which implies $cov(\X_1,\X_2)=0.$ Similarly, we show that  $cov(\X_i,\X_j)=0$ for $i\neq j$. Now we use  well known facts from the general theory of probability that
if a random vector has a multivariate normal distribution (joint normality), then any two or more of its components that are uncorrelated, are independent. This implies that any two or more of its components that are pairwise independent are independent.  
Normality of linear combinations 
$\sum_{i=1}^m a_i\X_i$ for all $a_i\in\R$, means
joint normality of $(\X_1,\dots, \X_m)$ (see e.g. the definition of multivariate normal law in Billingsley \cite{Billingsley}) and taking into account that random variables $\X_1,\dots, \X_m$ are pairwise uncorrelated, we obtain independence of   $\X_1,\dots,\X_m$. 
\end{proof}
The above theorem gives us the main result by Cook \cite{Cook} but under the additional assumption that all moments exist. Note that Cook does not assume any moments but assumes integrability.


\begin{cor}
Let $(\X_1,\dots, \X_m) \textrm{ and } (\X_{m+1},\dots,\X_n)$ be independent  random vectors  with all moments, where  $\X_i$ are
nondegenerate, and let statistic $\sum_{i=1}^n(\X_i+a_i)^2$
have a distribution which depends only on $\sum_{i=1}^n a_i^2$, $a_i\in \mathbb{R}$  and $1\leq m < n$.  Then $\X_i $ are independent and have the same normal distribution with zero means. 
\end{cor}
\begin{proof}
If we put $\Y=\sum_{i=1}^m \X_i^2$ and  $\Z=\sum_{i=m+1}^n \X_i^2$  in Theorem \ref{tw:Main1} then we get  
$$\sum_{i=1}^m a_i\X_i+\Y+\sum_{i=m+1}^na_i\X_i+\Z =\sum_{i=1}^n(\X_i+a_i/2)^2-\frac{1}{4} \times\sum_{i=1}^n a_i^2.$$
This means that the distribution of $\sum_{i=1}^m a_i\X_i+\Y+\sum_{i=m+1}^na_i\X_i+\Z$ depends only on $\sum_{i=1}^n a_i^2$, which by Theorem  \ref{tw:Main1} implies the statement. 
\end{proof}

 A simple modification of the above arguments can be applied to get the following proposition. In this  proposition we assume a bit more than in Proposition \ref{prop:1} and it offers a bit
stronger conclusion. 

\begin{prop}
Let $(\X,\Y) \textrm{ and } (\Z,\T )$ be independent  and nondegenerate random vectors  with all moments and let 
$ a\X+\Y+ b\Z+\T$ and $ \X+a\Y +\Z+b\T$  
have a distribution which depends only on $a^2+b^2$, $a,b\in \mathbb{R}$. Then $\X,\Y,\Z,\T $ are independent and have normal distribution with zero means and $Var(\X)=Var(\Z)$, $Var(\Y)=Var(\T)$.  
\end{prop}
\begin{proof}
From  Proposition \ref{prop:1} we get that $\X,\Y,\Z,\T $  have normal distribution with zero means and $Var(\X)=Var(\Z)$, $Var(\Y)=Var(\T)$. Now we will show that  $\X$ and $\Y $ are independent random variables.
We  proceed analogously to the proof of Proposition \ref{prop:1} with $h_k( a^2+b^2)=r_k( a\X+\Y+b\Z +\T)-r_k(\Y+\T)$, 
 which gives us equality \eqref{eq:2}, i.e. 
%

\begin{align} \sum_{i=1}^ka^{i} {k \choose i} r_k(\underbrace{ \X,\dots,\X}_{i-times},\underbrace{ \Y,\dots,\Y}_{k-i-times})=a^2(r_k(\X+\Y)-r_k(\Y)).\label{twr:pom1} \end{align} 
 for  $k\geq 2$.  From this we conclude  that 
\begin{align}  r_k(\underbrace{ \X,\dots,\X}_{i-times},\underbrace{ \Y,\dots,\Y}_{l-times})=0
\label{twr:pom3} \end{align} 
 for all $i,l\in\mathbb{N}$, $i\neq 2$. By a similar argument applied to a statistic  $ \X+a\Y +\Z+b\T$   we get  $$ r_k(\underbrace{ \X,\dots,\X}_{l-times},\underbrace{ \Y,\dots,\Y}_{i-times})=0$$  
for all
$i,l\in\mathbb{N}$, $i\neq 2$   
which together with \eqref{twr:pom3} gives us independence of $\X$ and $\Y$.  Independence of $\Z$ and $\T$ follows similarly.
\end{proof}

\begin{center} Open Problem and Remark
\end{center}

\noindent \textbf{Problem 1.} In Proposition \ref{prop:1} (Theorem \ref{tw:Main1}) in this paper we assume that random
variables have all moments. I thought it would be interesting to show that we can skip this assumption.  A
version of Proposition \ref{prop:1}, with integrability replaced by the assumption that $\S_1, \S_2$ have the
same law, can be deduced from known results (note that this version  does not imply Theorem \ref{tw:Main1}). 

Here we sketch the proof
of a version of Proposition \ref{prop:1} under reduced moment assumptions but we assume additionally that random variables $\S_1, \S_2$ have the same law. Assume that $\S_1, \S_2, \Y, \Z$ have finite
moments of some positive order $p\geq 3$. Then, for every $\epsilon>0$ and any two values of $p_i > 0$, where $i\in\{1,2\}$, we see that $E(a\S_1 + \epsilon\Y + b\S_2 + \epsilon\Z)^{p_i}$  is a function of $a^2+b^2$. Passing to the limit as $\epsilon \to 0$ and using homogeneity, we deduce that
$E(a\S_1  + b\S_2 )^{p_i} = K(a^2 + b^2)^{p_i/2}$ with $K = 2^{-p_i/2}E|\S_1 + \S_2|^{p_i} = E|\S_1|^{p_i} = E|\S_2|^{p_i}$. Since this
holds for any two  odd values  $0 < p_1 < p_2 < p$, by  Theorem 2 from \cite{Braverman} we see that $\S_1$ is normal (the reason  why we assume $p\geq 3$ is that Braverman \cite{Braverman}  assumed that $p_1, p_2$ are odd). 

\noindent \textbf{Problem 2.} 
 At the end it is worthwhile to mention  the  most important characterization which is true in noncommutative
and classical probability.  In free probability   Bo\.zejko, Bryc and Ejsmont  proved that the first conditional linear  moment and 
 conditional quadratic variances  characterize free Meixner laws (Bo\.zejko and Bryc \cite{BoBr}, Ejsmont \cite{Ejs}). 
Laha-Lukacs type  characterizations  of random variables in free probability are also studied by Szpojankowski,  Weso\l owski \cite{SzWes}.
 They give a characterization of noncommutative
free-Poisson and free-Binomial variables by properties of the first two conditional moments,
which mimics Lukacs-type assumptions known from classical probability. 
The article \cite{Kemp} studies  the asymptotic behavior of the Wigner integrals. Authors prove that a normalized sequence of multiple Wigner integrals (in a fixed order of free Wigner chaos) converges in law to the
standard semicircular distribution if and only if the corresponding
sequence of fourth moments converges to 2, the fourth moment of
the semicircular law.  This finding extends the recent results by Nualart and Peccati \cite{Nualart} to free probability theory. 

 At this point it is worth mentioning \cite{Hiwatashi}, where the Kagan-Shalaevski characterization for free random variable was shown. 
 It would be worth asking whether the Theorem \ref{tw:Main1} is true in free probability theory.
 Unfortunately the above proof of Theorem \ref{tw:Main1} doesn't work in free probability theory, because free cumulates are noncommutative. But the Proposition \ref{prop:1} is true in free probability, with nearly the same proof (thus we can only get that $\X_i$ has  free normal distribution  under the assumption of Theorem \ref{tw:Main1}). 

\begin{center} Acknowledgments
\end{center}
The author would like to thank  M. Bo\.zejko  for several discussions and helpful comments during 
the preparation of this paper. 
The work was partially supported by  the OPUS grant DEC-2012/05/B/ST1/00626 of National
Centre of Science and  by the Austrian Science Fund (FWF) Project No P 25510-N26.

\end{document}